\documentclass{nm}
\renewcommand\thisyear {200x}

\renewcommand\thisvolume{xx}

\usepackage{charter}
\usepackage[charter]{mathdesign}

\usepackage{bm}
\usepackage{dsfont}
\usepackage{amsmath}
\usepackage{amsthm}
\usepackage{mathtools}
\usepackage{amsfonts}
\usepackage{graphicx}
\usepackage[usenames,dvipsnames]{xcolor}
\usepackage{bbm}

\newcommand{\MG}[1]{{\color{black}{#1}}}

\newcommand{\mcB}{\mathcal{B}}

\newcommand{\mcD}{\mathcal{D}}
\newcommand{\mcG}{\mathcal{G}}
\newcommand{\mcDw}{\mathcal{D}_\omega}
\newcommand{\mcGw}{\mathcal{G}_\omega}
\newcommand{\mcLw}{\mathcal{L}_\omega}

\newcommand{\mcF}{\mathcal{F}}

\newcommand{\mcL}{\mathcal{L}}

\newcommand{\mbR}{\mathbb{R}}

\newcommand{\mbRn}{{\mathbb{R}^n}}

\newcommand{\omg}{{\Omega}}

\newcommand{\omgomgi}{{\Omega\cup\Omega_I}}

\def \alphab{{\boldsymbol\alpha}}

\def \vb{\mathbf{v}}
\def \xb{\boldsymbol{x}}
\def \yb{\boldsymbol{y}}
\def \zb{\bm{z}}

\def \intl{\int\limits}
\def \gameq{\gamma_{\textup{eq}}}

\newcommand{\vertiii}[1]{{\left\vert\kern-0.25ex\left\vert\kern-0.25ex\left\vert #1 
    \right\vert\kern-0.25ex\right\vert\kern-0.25ex\right\vert}}

\begin{document}


\markboth{M. D'Elia and M. Gulian}{Analysis of anisotropic fractional diffusion}
\title{Analysis of Anisotropic Nonlocal Diffusion Models: Well-posedness of Fractional Problems for Anomalous Transport}

\author[Marta D'Elia and Mamikon Gulian]{Marta D'Elia\corrauth \affil{1}$^{,*}$ and Mamikon Gulian\affil{2}}
\address{\affilnum{1}\ Computational Science and Analysis, Sandia National Laboratories, CA\\
\affilnum{2}\ Center for Computing Research, Sandia National Laboratories, NM}
%
%
\emails{{\tt mdelia@sandia.gov} (M. D'Elia), {\tt mgulian@sandia.gov} (M. Gulian)}
%



\begin{abstract}
We analyze the well-posedness of an anisotropic, nonlocal diffusion equation. 
Establishing an equivalence between weighted and unweighted anisotropic nonlocal diffusion operators in the vein of unified nonlocal vector calculus, we apply our analysis to a class of fractional-order operators and present rigorous estimates for the solution of the corresponding anisotropic anomalous diffusion equation. Furthermore, we extend our analysis to the anisotropic diffusion-advection equation and prove well-posedness for fractional orders $s\in[0.5,1)$. We also present an application of the advection-diffusion equation to anomalous transport of solutes. 
\end{abstract}

\maketitle

\section{Introduction}\label{sec:intro}
Nonlocal models have become a preferred modeling choice for scientific and engineering applications featuring global behavior that is affected by small scales. In particular, nonlocal models can capture effects that classical partial differential equations (PDEs) fail to describe; these effects include multiscale behavior and anomalous transport such as superdiffusion and subdiffusion. A nonlocal equation is characterized by integral operators acting on a lengthscale or ``horizon'', describing long-range forces and reducing the regularity requirements on the solutions. Engineering applications include surface or subsurface transport 
\cite{Benson2000,Benson2001,Deng2004,katiyar2019general,katiyar2014peridynamic,Schumer2003,Schumer2001},
fracture mechanics
\cite{Ha2011,Littlewood2010,Silling2000},
turbulence
\cite{DiLeoni-2020,Pang2020},
image processing
\cite{Buades2010,DElia2019imaging,Gilboa2007,Lou2010}
and stochastic processes
\cite{Burch2014,DElia2017,Meerschaert2012,Metzler2000,Metzler2004}.

In this work we focus mainly on nonlocal operators of fractional type, of the form
\begin{equation}\label{eq:baby_operator}
\mcL_{\omega;A}=\mcDw\left(A(\xb)\mcGw\right), \text{ with }\omega \propto |\xb-\yb|^{-n-s}
\end{equation}
and their use in models for solute transport. As reviewed carefully in Section \ref{sec:notation}, $A(\xb)$ denotes a diffusion tensor and $\mcDw$ and $\mcGw$ denote weighted nonlocal divergence and gradient operators, respectively, with weight function $\omega$. Modeling and simulation of surface and subsurface transport is challenging due to the inevitable heterogeneities of the media which generate, at the continuum scale, diffusion processes that exhibit transport rates which may be faster or slower than those described by the classical integer-order diffusion equation. While at the smaller scales a local PDE model may be able to accurately describe diffusion processes by explicitly embedding the heterogeneities in the model parameters, at the continuum scale, such models may fail to do so. In contrast, a fractional-order model using an diffusion operator of the form \eqref{eq:baby_operator} may act as a homogenized model that encodes the heterogeneities of the medium in the integral operator itself. Several fractional models have been proposed in the literature for such applications; we refer to \cite{Sun2020} for an extensive review. 
The problems we study in the present work serve as a novel models for anisotropic, anomalous transport. 

Though significant progress has been made in the analysis and simulation of nonlocal equations, many modeling and computational challenges still remain. Among those, we mention the expensive computational cost \cite{AinsworthGlusa2018,Capodaglio2020DD,DElia-ACTA-2020,DEliaFEM2020,Pasetto2019,Silling2005meshfree,Wang2010}, the identification of optimal model parameters or kernel functions \cite{burkovska2020,DElia2014DistControl,DElia2016ParamControl,Gulian2019,Pang2019fPINNs,Pang2017discovery,Xu2020learning,You2020Regression,You2020aaai}, and gaps and open questions within the nonlocal and fractional calculus theories \cite{DElia2020Helmholtz,Shieh2015,Shieh2017,silhavy2020fractional}. One focus of this work is on such theoretical gaps in the analysis of the so-called weighted nonlocal operators, their connection to fractional operators and on the well-posedness of the corresponding diffusion problem. The connection between nonlocal and fractional operators, investigated for the first time in \cite{Defterli2015} and \cite{DElia2013}, was studied extensively in the recent work \cite{DElia2020Unified}. There, the authors introduce the notion of a {\it unified nonlocal vector calculus} for scalar functions and introduce a universal nonlocal Laplace operator that includes, as a special cases, the well-known fractional Laplace operators and other variants of the latter. Variational results of the unified calculus allow the extension of the well-established theory of unweighted operators \cite{Du2012} to weighted nonlocal operators and, more specifically, operators of fractional-order vector calculus. 

We continue this effort by providing conditions for well-posedness for problems involving anisotropic nonlocal weighted operators, with a special focus on fractional operators of the form \eqref{eq:baby_operator} for which we establish well-posedness. 

Our major contributions are:
\vspace{-.3cm}
\begin{enumerate}
    \item The extension of results presented in \cite{DElia2020Unified} to the parabolic case.
    \vspace{-.3cm}
    \item The well-posedness analysis of the elliptic and parabolic anisotropic nonlocal diffusion problem and, for the fractional case, the well-posedness of such problems for fractional orders $s\in(0,1)$.
    \vspace{-.3cm}
    \item The well-posedness analysis of an anisotropic nonlocal transport (advection-diffusion) equations and, for the fractional case, the well-posedness of such problem for fractional orders $s\in[0.5,1)$. \label{item:advection}
    \vspace{-.3cm}
    \item The application of (\ref{item:advection}) to surface and subsurface solute transport. 
\end{enumerate}

\paragraph{Outline of the article} In Section \ref{sec:notation} we recall important results of the unified nonlocal calculus and we extend the well-posedness analysis presented in \cite{DElia2020Unified} to the time-dependent case. In Section \ref{sec:anisotropic-Poisson} we introduce the anisotropic Poisson equation involving the operator \eqref{eq:baby_operator}, introduce an equivalent unweighted diffusion operator by proving the existence of a symmetric {\it equivalence kernel}, and extend variational results of \cite{DElia2020Unified} to the anisotropic case. In the same section we also rigorously prove the well-posedness of the corresponding fractional-order problem. In Section \ref{sec:anisotropic-parabolic} we introduce the anisotropic parabolic problem for \eqref{eq:baby_operator} and analyze its well-posedness; we also provide an \emph{a priori} estimate for the its solution and specialize it for the fractional case. Section \ref{sec:solute} deals with anomalous transport of solutes in the surface or subsurface. Here, we prove the well-posedness of the anisotropic, anomalous advection-diffusion problem for both the general and fractional case. In Section \ref{sec:conclusion} we summarize our results and contributions. 

\section{Notation and previous results}\label{sec:notation}

In this section we recall the definition of weighted and unweighted nonlocal operators with special emphasis on the case of fractional kernels. This review is based on Section 2 of \cite{DElia2020Unified}; as we recall those results, we also extend them to the time-dependent, parabolic case.

\subsection{Unweighted operators and corresponding volume-constrained problems}
Let $\alphab:\mathbb{R}^n\times\mathbb{R}^n\to\mathbb{R}^n$, for $n=1,2,3$, be an anti-symmetric two-point vector function. For $\vb:\mathbb{R}^n\times\mathbb{R}^n\to\mathbb{R}^n$, the nonlocal {\it unweighted divergence} $\mathcal{D}\vb:\mathbb{R}^n\to\mathbb{R}$ is defined as
\begin{equation}\label{eq:unw-div}
\begin{aligned}
\mathcal{D}\vb(\xb) := \int_{\mathbb{R}^n} (\vb(\xb,\yb)+\vb(\yb,\xb))\cdot\alphab(\xb,\yb)d\yb. 
\end{aligned}
\end{equation}
For $u:\mathbb{R}^n\to\mathbb{R}$ the nonlocal {\it unweighted gradient}, $\mathcal G u:\mathbb{R}^n\times\mathbb{R}^n\to\mathbb{R}^n$, the negative adjoint of \eqref{eq:unw-div} \cite{Du2013}, is defined as
\begin{equation}\label{eq:unw-grad}
\mathcal{G}u(\xb,\yb) = (u(\yb)-u(\xb))\alphab(\xb,\yb).
\end{equation}
The nonlocal {\it unweighted Laplacian} is defined as the composition of unweighted nonlocal divergence and gradient, i.e.
\begin{equation}\label{eq:unw-lapl}
\mcL u(\xb) = \mcD\mcG u(\xb) =
2\int_\mbRn (u(\yb)-u(\xb))\gamma(\xb,\yb)d\yb,
\end{equation}
where the nonnegative kernel\footnote{For a discussion on sign-changing kernels and nonsymmetric kernels, see  \cite{mengesha2013analysis} and \cite{DElia2017}, respectively.} $\gamma$ is given by $\gamma=\alphab\!\cdot\!\alphab$.

In order to define a diffusion problem in a bounded domain $\Omega\subset\mathbb R^n$, by definition of $\mcL u(\xb)$, it is necessary to evaluate $u(\xb)$ for $\xb \in \mbRn\setminus\Omega$. We refer to conditions on $u$ in the exterior of the domain as exterior conditions or volume constraints. With this in mind, the strong form of an unweighted nonlocal diffusion problem is given by: for $f:\Omega\to\mbR$, $u_0:\Omega\to\mbR$ and $g:\mbRn\setminus\Omega\to\mbR$, find $u$ such that
\begin{equation}\label{eq:bound-truncated-prob}
\left\{\begin{aligned}
\partial_t u(\xb,t) &= \mcL u (\xb,t) + f (\xb,t),  
&\quad (\xb,t)\in\Omega\times(0,T)\\
u(\xb,t) &= g(\xb,t),         
&\quad (\xb,t)\in\mbRn\setminus\Omega\\
u(\xb,0) &= u_0(\xb),  
&\quad \xb\in\Omega
\end{aligned}\right.
\end{equation}
where the second condition in \eqref{eq:bound-truncated-prob} is the nonlocal counterpart of a Dirichlet boundary condition for PDEs and it is referred to as {\it Dirichlet volume constraint}\footnote{For definition and analysis of Neumann volume constraints we refer to \cite{Du2012} and for \MG{their} numerical treatment we refer to, e.g., \cite{DEliaNeumann2019}.} or an \emph{exterior value condition}. The work \cite{DElia2017} shows that such condition is required to guarantee the well-posedness of \eqref{eq:bound-truncated-prob}. For simplicity and without loss of generality we analyze the homogeneous case $g=0$; all the results below can be extended to the non-homogeneous case using ``lifting'' arguments (see, e.g., \cite{DElia2014DistControl}).

To obtain the variational form of equation \eqref{eq:bound-truncated-prob}, we apply the following nonlocal form of the first Green's identity, introduced in \cite{Du2013}\footnote{Note that \cite{Du2013} introduces the first Green's identity for operators whose kernels have support $B_\delta(\xb)$, for $\delta>0$. Equation \eqref{eq:unweighted-Green} corresponds to the case $\delta=\infty$, of interest in this work. The same result, solely for fractional operators, was also proved in \cite{Dipierro2017}.}:
\begin{equation} \label{eq:unweighted-Green}
  \int\limits_\Omega -\mcL u(\xb) \, v(\xb) \,d\xb  
= 
\MG{\int_{\mathbb{R}^n} \int_{\mathbb{R}^n}}
\mcG u(\xb,\yb) \cdot \mcG v(\xb,\yb) \,d\yb\,d\xb 
+ \int\limits_{\mbRn\setminus\Omega} \mcD(\mcG u)(\xb)\,v(\yb) \,d\xb.
\end{equation}
Multiplying \eqref{eq:bound-truncated-prob} by a test function $v$ such that $v = 0$ on $\mbRn\setminus\Omega$ and integrating over the domain $\Omega$ yields\MG{, for all $t \ge 0$,}
\begin{align}\label{eq:bound-truncated-weak}
0 &=  \int\limits_\Omega (\partial_t u\MG{(\xb,t)}-\mcL u\MG{(\xb,t)} -f\MG{(\xb,t)})\, v\MG{(\xb)} \,d\xb
\\
&\begin{multlined}[t]
= \int\limits_\Omega \partial_t u\MG{(\xb,t)}\,v\MG{(\xb)} \,d\xb +
\MG{\int_{\mathbb{R}^n} \int_{\mathbb{R}^n}}
\mcG u\MG{(\xb,\yb,t)} \cdot \mcG v\MG{(\xb,\yb)} \,d\yb\,d\xb
\\
+ \int\limits_{\mbRn\setminus\Omega} \mcD(\mcG u)\MG{(\xb,t)}\,v\MG{(\xb)} \,d\xb
- \int\limits_\Omega f\MG{(\xb,t)}\, v\MG{(\xb)} \,d\xb,
\end{multlined}
\end{align}
where the integral over $\mbRn\setminus\Omega$ on the right-hand side is zero due to the properties of $v$. 
Given a function space $S$ with norm $\|\cdot\|$, we define the space $L^2(0,T;S)$ as follows
\begin{equation*}
L^2(0,T;S) = \{
\MG{w:\mathbb{R}^n \times \mathbb{R} \rightarrow \mathbb{R} \text{ such that } w(\cdot,t) \in S \text{ $\forall$ $t \ge 0$, and }}
\|\MG{w}(\cdot,t)\|_S\in L^2(0,T)\}.
\end{equation*}
Then, the weak form of the nonlocal diffusion problem reads as follows. For $f\in L^2(0,T;V_\Omega'(\mbRn))$, find $u\in L^2(0,T;V_\Omega(\mbRn)\MG{)}$ such that
\begin{equation}\label{eq:bound-truncated-weak-forms}
(\partial_t u,v)+ \mcB(u,v) = \mcF(v), \;\;\forall\, v\in V_\Omega(\mbRn),
\end{equation}
where $(\cdot,\cdot)$ indicates the $L^2$ inner product over $\Omega$, and
\begin{equation}\label{eq:A-F-V}
\begin{aligned}
\mcB(u,v) & = 
\MG{\int_{\mathbb{R}^n}\int_{\mathbb{R}^n}}
\mcG u\MG{(\xb,\yb)} \cdot \mcG v\MG{(\xb,\yb)} \,d\yb\,d\xb,\\[2mm]
\mcF(v) & =\int_\Omega f\MG{(\xb)}\, 
v\MG{(\xb)} \,d\xb,\\[2mm]
V_\Omega(\mbRn) & =\{v\in L^2(\mbRn): \vertiii{v}<\infty\;
{\rm and} \; v|_{\mbRn\setminus\Omega} = 0\}.
\end{aligned}
\end{equation}
Here, the {\it energy norm} $\vertiii{\cdot}$ is defined as
\begin{equation}\label{eq:unweighted-energy}
\vertiii{v}^2 = 
\MG{\int_{\mathbb{R}^n} \int_{\mathbb{R}^n}}
|\mcG v\MG{(\xb,\yb)}|^2\,d\yb\,d\xb,
\end{equation}
and the space $V_\Omega'$ is the dual space of $V_\Omega$.
Note that the bilinear form $\mcB(\cdot,\cdot)$ defines an inner product on $V_\Omega(\mbRn)$ and that $\vertiii{u}^2= \mcB(u,u)$. This fact implies that the bilinear form is coercive and, hence, weakly coercive. Together with the continuity of $\mcB$ and $\mcF$, this yields the well-posedness of the weak form \eqref{eq:bound-truncated-weak-forms}\cite{DElia2017}.

\subsection{Weighted operators and corresponding volume-constrained problems}
\label{sec:intro_weighted_operators}
We let $\omega:\mathbb{R}^n\times\mathbb{R}^n\to\mathbb{R}$ be a nonnegative, symmetric scalar function known as the \emph{weight} function. For $\vb:\mathbb{R}^n\to\mathbb{R}^n$, the nonlocal {\it $\omega$-weighted divergence} $\mcD_\omega\vb:\mathbb{R}^n\to\mathbb{R}$ is defined as
\begin{align}
\begin{split}
\label{eq:w-div}
\mathcal{D}_{\omega} \vb(\xb) &:= \mcD(\omega(\xb,\yb)\vb(\xb))\\
&= 
\int_{\mathbb{R}^n} (\omega(\xb,\yb)\vb(\xb) + \omega(\yb,\xb)\vb(\yb))\cdot\alphab(\xb,\yb)d\yb.
\end{split}
\end{align}
For $u:\mathbb{R}^n\to\mathbb{R}$, the nonlocal {\it $\omega$-weighted gradient}, negative adjoint of the divergence \cite{Du2013}, $\mcG_\omega u:\mbRn\to\mbRn$ is defined as
\begin{align}\label{eq:w-grad}
\begin{split}
\mcG_\omega u(\xb) &:= \int_\mbRn
\MG{\omega(\xb,\yb)}
\mcG u(\xb,\yb) d\yb  \\
&= 
\int_\mbRn
\MG{\omega(\xb,\yb)}
(u(\yb)-u(\xb))\alphab(\xb,\yb)
d\yb .
\end{split}
\end{align}
As in the unweighted case, we define the nonlocal {\it $\omega$-weighted Laplacian} as the composition of \eqref{eq:w-div} and \eqref{eq:w-grad}, i.e., 
\begin{equation}\label{eq:w-lapl}
\begin{aligned}
\mcLw u(\xb) &= \mcD_\omega\mcG_\omega u(\xb)\\
& 
\begin{multlined}[t]
=\int_\mbRn  \left[
\MG{\omega(\xb,\yb)}
\int_\mbRn
\MG{\omega(\xb,\zb)}
(u(\zb)-u(\xb))\alphab(\xb,\zb)
d\zb  \right.\\ \qquad +
\MG{\omega(\yb,\xb)}
\left.\int_\mbRn
\MG{\omega(\yb,\zb)}
(u(\zb)-u(\yb))\alphab(\yb,\zb)
d\zb 
\right] \cdot\alphab(\xb,\yb)d\yb.
\end{multlined}
\end{aligned}
\end{equation}
\MG{Using the symmetry of $\omega$, we can further write
\begin{multline}\label{eq:w-lapl-with-symmetry}
\mcLw u(\xb)
=\int_\mbRn
\omega(\xb,\yb)
 \left[
\int_\mbRn
\omega(\xb,\zb)
(u(\zb)-u(\xb))\alphab(\xb,\zb)
d\zb  \right.\\ +
\left.\int_\mbRn
\omega(\yb,\zb)
(u(\zb)-u(\yb))\alphab(\yb,\zb)
d\zb 
\right] \cdot\alphab(\xb,\yb)d\yb.
\end{multline}}

\MG{\begin{remark}
We have reviewed the definitions \eqref{eq:w-div}, \eqref{eq:w-grad}, and \eqref{eq:w-lapl} assuming that $\omega$ is a symmetric, scalar valued function. In Section \ref{sec:anisotropic-Poisson}, we consider the case when $\omega$ is a nonsymmetric tensor. Definitions \eqref{eq:w-div}, \eqref{eq:w-grad}, and \eqref{eq:w-lapl}, but not the simplification \eqref{eq:w-lapl-with-symmetry}, may be utilized for this case with products of $\omega$ and vectors being interpreted as matrix-vector multiplication. 
\end{remark}}

As for the unweighted case, problems defined on bounded domains involving these operators require a volume constraint on the exterior of $\Omega$.
We introduce the strong form of a weighted, nonlocal diffusion problem with homogeneous volume constraints. For $f:\Omega\to\mbR$ and $u_0:\Omega\to\mbR$, find $u$ such that
\begin{equation}\label{eq:weight-bound-truncated-prob}
\left\{\begin{aligned}
\partial_t u(\xb,t) &= \mcLw u(\xb,t) + f(\xb,t),  &\quad (\xb,t)\in\Omega\times(0,T]\\
u(\xb,t) &= 0,       &\quad (\xb,t)\in\mbRn\setminus\Omega\times(0,T]\\
u(\xb,0) &= u_0(\xb),  &\quad \xb\in\Omega
\end{aligned}\right.
\end{equation}
where the second condition in \eqref{eq:weight-bound-truncated-prob} is still referred to as Dirichlet volume constraint. Next, by multiplying \eqref{eq:weight-bound-truncated-prob} by a test function \MG{$v:\mathbb{R}^n \rightarrow \mathbb{R}$ such that}
\begin{equation}\label{eq:exterior_cond_v}
v=0 \text{ in } \mbRn\setminus\Omega,
\end{equation}
and integrating over the domain $\Omega$, we have the following weak form:
\begin{equation}\label{eq:weight-truncated-weak}
\int_\Omega (\partial_t u\MG{(\xb,t)}-\mcLw u\MG{(\xb,t)} -f\MG{(\xb,t)})\, v\MG{(\xb)} \,d\xb = 0,
\MG{\text{ for all $t > 0$.}}
\end{equation}
The work \cite[Theorem~5.1]{DElia2020Unified} introduced the following weighted nonlocal Green's first identity
\begin{equation}\label{eq:weighted-green}
 \int_{\Omega}-\mcLw u(\xb) v(\xb) \ d\xb
= \int_\mbRn \mcGw u(\xb) \!\cdot\! \mcGw u(\xb) d\xb 
+ \int_{\mbRn \setminus \Omega} \mcDw\mcGw u(\xb) v(\xb) \ d\xb.
\end{equation}
By substituting the latter in \eqref{eq:weight-truncated-weak}, we obtain
\begin{multline}
\int_\Omega \partial_t u\MG{(\xb,t)}\,v\MG{(\xb)} \,d\xb +
\int_\mbRn \mcGw u\MG{(\xb,t)} \!\cdot\! \mcGw u\MG{(\xb,t)} d\xb \\
+ \int_{\mbRn \setminus \Omega} \mcDw\mcGw u\MG{(\xb,t)} v(\xb) \ d\xb
- \int_\Omega f\MG{(\xb,t)}\, v\MG{(\xb)} \,d\xb=0.
\end{multline}
By \eqref{eq:exterior_cond_v}, the integral over $\mbRn\setminus\Omega$ on the left-hand side is zero.
Thus, the weak form of the nonlocal diffusion problem reads as follows. For $f\in L^2(0,T;(V^\omega_\Omega)'(\mbRn))$, and $u_0\in V^\omega_\Omega(\mbRn)$, find $u\in L^2(0,T;V^\omega_\Omega(\mbRn)$ such that
\begin{equation}\label{eq:weighted-weak}
(\partial_t u,v)+ \mcB_\omega(u,v) = \mcF(v), \;\;\forall\, v\in V^\omega_\Omega(\mbRn),
\end{equation}
where 
\begin{equation}\label{eq:A-F-V-w}
\begin{aligned}
\mcB_\omega(u,v) & = \int_\mbRn \mcGw u(\xb) \!\cdot\! \mcGw \MG{v}(\xb) d\xb,\\[2mm]
V^\omega_\Omega(\mbRn) & =\{v\in L^2(\mbRn): \vertiii{v}_\omega<\infty\;
                {\rm and} \; v|_{\mbRn\setminus\Omega} = 0\},
\end{aligned}
\end{equation}
and where the {\it weighted energy} is defined as
\begin{equation}\label{eq:weighted-energy}
\vertiii{v}_{\omega}^2 = \int_\mbRn 
\MG{|\mcGw v(\xb)|^2}\,d\xb.
\end{equation}
The well-posedness of problem \eqref{eq:weighted-weak} follows when an equivalence relationship can be established between weighted and unweighted operators, as we summarize in the next section. 

\subsection{The equivalence kernel}
The equivalence theorem between weighted and unweighted operators proved in \cite{DElia2020Unified} provides an equivalence kernel $\gameq$ that, for given $\alphab$ and $\omega$, guarantees that $\mcL=\mcL_\omega$. In what follows, we summarize the main result and its consequences. 

\begin{theorem}\label{thm:equivalence}\cite[Theorem~4.1]{DElia2020Unified}
Let $\mcDw$ and $\mcGw$ be the operators associated with the symmetric \MG{scalar} weight function $\omega$ and the anti-symmetric function $\alphab$. For the equivalence kernel $\gameq$ defined by
\begin{equation}\label{eq:kerndef}
\begin{aligned}
2\gameq(\xb,\yb;\omega,\alphab)
&= \int_\mbRn [\MG{\omega(\xb,\yb)}\alphab(\xb,\yb)\cdot 
   \MG{\omega(\xb,\zb)}\alphab(\xb,\zb)\\
&  \hspace{1cm}+ \MG{\omega(\zb,\yb)}\alphab(\zb,\yb)\cdot
   \MG{\omega(\xb,\yb)}\alphab(\xb,\yb)\\[2mm]
&  \hspace{1cm}+\MG{\omega(\zb,\yb)}\alphab(\zb,\yb)
   \cdot\MG{\omega(\xb,\zb)}\alphab(\xb,\zb)]d\zb,
\end{aligned}
\end{equation}
the weighted operator $\mcLw = \mcDw\mcGw$ and the unweighted Laplacian operator $\mcL$ with kernel $\gamma_{\textup{eq}}$ are equivalent, i.e.
$\mcL = \mcLw.$
\end{theorem}

This result, and the weighted nonlocal Green's first identity, imply the following important variational equivalence.
\begin{theorem}\label{thm:var-equivalence}\cite[Theorem~5.2]{DElia2020Unified}
For $\gameq(\xb,\yb;\omega,\alphab)$ defined as in \eqref{eq:kerndef}, the variational forms associated with weighted and unweighted nonlocal operators are equivalent. That is, for all $v=0$ in $\mbRn\setminus\Omega$,
\begin{equation}\label{eq:weak-equivalence}
 \mcB(u,v)
=\int_{\mbRn}\int_{\mbRn}\mcG u\MG{(\xb,\yb)} \MG{\cdot} \mcG v\MG{(\xb,\yb)} \,d\yb\,d\xb 
=\int_{\mbRn} \mcGw u\MG{(\xb)} \MG{\cdot} \mcGw v\MG{(\xb)} \,d\xb
=\mcB_\omega(u,v).
\end{equation}
\end{theorem}

An immediate consequence of this theorem is the equivalence of weighted and unweighted energies, i.e.
\begin{equation}\label{eq:equivalence_of_norms}
\vertiii{v}^2 = \mcB(v,v) = \mcB_\omega(v,v)= \vertiii{v}^2_\omega.
\end{equation}

More importantly, the variational equivalence allows us to extend the unweighted well-posedness results to the weighted case, anytime the equivalence kernel $\gameq$ induces an unweighted coercive bilinear form $A(\cdot,\cdot)$, see \cite{DElia2020Unified} for a discussion.

\subsection{The special case of fractional operators}
In this work, we are interested in the case of fractional operators. In this section we recall their definition and specify the choices of $\alphab$ and $\omega$ for which the weighted fractional Laplacian is equivalent to the standard fractional Laplacian. 

The (Riesz) fractional Laplacian is defined as \cite{Lischke2020}
\begin{equation}
\label{eq:riesz_Laplacian_rn}
(-\Delta)^s u = C_{n,s}
\int_{\mathbb{R}^n} \frac{u(\xb) - u(\yb)}{|\xb-\yb|^{n+2s}}d\yb,
\end{equation}
where
\begin{equation}\label{eq:frac_laplc_const}
C_{n,s}=
\frac{4^{s} \Gamma\left(s+\frac{n}{2}\right)}
{\pi^{n/2}|\Gamma(-s)|}.
\end{equation}
The weighted fractional gradient and divergence operators are defined as \cite{Mazowiecka2018,Ponce2016,Shieh2015,Shieh2017,silhavy2020fractional}
\begin{align}\label{eq:frac-Cart}
\begin{split}
{\rm grad}^s  u(\xb) &= 
\int_{\mbRn}\left[ u(\xb) - u(\yb) \right]
\frac{\xb-\yb}{|\xb-\yb|}
\frac{1}{|\xb-\yb|^{n+s}}
d\yb,
\\
{\rm div}^s \vb(\xb) &= 
\int_{\mbRn}
\left[ \vb(\xb) - \vb(\yb) \right]
\cdot
\frac{\xb-\yb}{|\xb-\yb|}
\frac{1}{|\xb-\yb|^{n+s}}
d\yb.
\end{split}
\end{align}
We summarize in the following theorem several results proved in \cite{DElia2020Unified}.

\begin{theorem}\label{thm:fractional_special_case_nonlocal}
Let $\vb \in {\bf H}^s(\mathbb{R}^d)$ and $u \in H^s(\mathbb{R}^d)$.
For the weight function and kernel
\begin{equation}\label{eq:equivalence-w-alpha}
\begin{array}{l}
\displaystyle
\omega= C_\omega |\xb-\yb|^{-(n+s)}, \quad
\alphab(\xb,\yb) = \frac{\yb-\xb}{|\yb-\xb|},
\end{array}
\end{equation}
where $C_\omega$ is the constant\footnote{The constant $C_\omega$ may be expressed as $G_s/\sqrt{-D_{n,s}}$ in the notation of \cite{DElia2020Unified}.} defined as \cite{DElia2020Unified}
\begin{equation}\label{eq:Comega}
C_\omega = \dfrac{2s \sin(\pi s/2)}{\Gamma(1-s)}
\int_{|\bm{\theta}| = 1, \bm{\theta}_1 \ge 0} 
|\bm{\theta}_1|^{s+1}
d\bm{\theta}
\end{equation}
the fractional divergence and gradient operators can be identified with the weighted nonlocal operators, 
\begin{equation}\label{eq:polar-Cart-weighted-equivalence}
\begin{aligned}
{\rm div}^s \vb(\xb) & = \mcDw \vb(\xb)\\[2mm]
{\rm grad}^s  u(\xb) & = \mcGw u(\xb).
\end{aligned}
\end{equation}
Furthermore, $\alphab(\xb,\yb)\omega(\xb,\yb) = (\yb-\xb)|\yb-\xb|^{-(n+s+1)}$, implies that
\begin{equation}\label{eq:frac-kernel-equivalence}
\gameq(\xb,\yb)=\gamma_{F\!L}(\xb,\yb)=
-\dfrac{C_{n,s}}{2}
\dfrac{1}{|\xb-\yb|^{n+2s}},
\end{equation}
where $F\!L$ stands for ``fractional Laplacian'' and $C_{n,s}$ is the defined as in \eqref{eq:frac_laplc_const}. Then, for $u \in H^{2s}(\mathbb{R}^n)$,
\begin{equation}
\mcL u= \mcLw u= -(-\Delta)^s u.
\end{equation}
\end{theorem}

In words, the fractional gradient and divergence are special instances of weighted gradient and divergence operators, for special choices of $\alphab$ and $\omega$, and their composition is equivalent to the standard fractional Laplacian operator.

\begin{remark}\label{rem:fractional-forms}
The corresponding weighted and unweighted diffusion problems are both well-posed in $L^2(0,T;H^s_\Omega(\mbRn))$ where $H^s_\Omega(\mbRn)=\{v\in H^s(\mbRn): v|_{\mbRn\setminus\Omega}=0\}$. This follows from the coercivity of $\mcB(\cdot,\cdot)$ for $\gamma=\gamma_{F\!L}$ \cite{DElia2020Unified} and from the variational equivalence in Theorem \ref{thm:var-equivalence}. More precisely, on one hand, the fact that the bilinear form $\mcB(\cdot,\cdot)$ associated with $\gamma_{FL}$ defines an inner product on $H^s_\Omega(\mbRn)$ guarantees the well-posedness of the unweighted parabolic problem. On the other hand, the variational equivalence guarantees that the weighted bilinear form $\mcB_\omega(\cdot,\cdot)$ associated with $\omega$ and $\alphab$ defined as in \eqref{eq:equivalence-w-alpha} is equivalent to $\mcB(\cdot,\cdot)$. This fact implies that the weighted parabolic problem is also well-posed in $L^2(0,T;H^s_\Omega(\mbRn))$.   
\end{remark}

\section{Well-posedness of anisotropic nonlocal Poisson problem}\label{sec:anisotropic-Poisson}
In this section we focus on the elliptic equation for the more general case in which the diffusion operator is anisotropic. This case corresponds to the introduction of a space-dependent diffusion tensor for which the theory reviewed in Section \ref{sec:notation} is not sufficient to guarantee existence and uniqueness of solutions. The analysis conducted in this section generalizes several results proved in \cite{DElia2020Unified} to the anisotropic case; these include existence of a symmetric equivalence kernel and a generalized Green's first identity for operators of the form $\mcDw\left(A(\xb)\mcGw\right)$, and a variational inequality that guarantees the well-posedness of the corresponding anisotropic volume constrained problem under certain conditions. We utilize this to prove well-posedness for the specific case of fractional operators. 

\subsection{Equivalence kernels and Green's identity for anisotropic weighted nonlocal operators}
We first introduce the anisotropic diffusion tensor and the corresponding nonlocal operator: let 
\begin{equation}\label{eq:Atensor}
A: \mathbb R^n \to \mathbb R^n\times\mathbb R^n\; \hbox{be bounded, measurable, symmetric and elliptic,}
\end{equation}
i.e. there exist $0 < \lambda_{\rm min}\leq\lambda_{\rm max}<\infty$ such that for all $\vb\in\mbRn$ and $\xb\in\mbRn$, 
\begin{equation}
\lambda_{\rm min}|\vb|^2\leq \vb\!\cdot\!A(\xb)\vb\leq \lambda_{\rm max} |\vb|^2.
\end{equation}
This implies the existence of a tensor-valued function $A^{\frac{1}{2}}(\xb)$ such that $A^{\frac{1}{2}}(\xb)A^{\frac{1}{2}}(\xb) = A(\xb)$.
We define the anisotropic nonlocal weighted Laplacian as
\begin{equation}\label{eq:ALaplacian}
\mcL_{\omega;A} u(\xb) = \mcDw(A(\xb)\mcGw u(\xb)).
\end{equation}
The following lemma shows that the tensor $A$ can be included in the weight function $\omega$ so that $\mcL_{\omega;A}$ can be equivalently written as $\mcL_{\widetilde\omega}$ for $\widetilde\omega=A^\frac12\omega$. It is important to note that, by construction, the weight function $\widetilde\omega$ is nonsymmetric, i.e. $\widetilde\omega(\xb,\yb)\neq\widetilde\omega(\yb,\xb)$, unless $A(\xb)= \hbox{\it Const.}$, and that $\widetilde\omega$ is also a tensor. In Section \ref{sec:intro_weighted_operators}, we assumed that the weight was symmetric, but the same operators and may be defined for nonsymmetric weight, and we utilize the same notation, e.g., for the diffusion operator $\mcL_{\widetilde\omega}$ for the diffusion operator \eqref{eq:w-lapl} with nonsymmetric weight $\widetilde\omega$.
\begin{lemma}\label{tildeomega}
Let $A$ satisfy \eqref{eq:Atensor} and $\alphab$ and $\omega$ be an anti-symmetric vector function and a symmetric scalar function respectively. Then, for $\widetilde\omega=A^\frac12\omega$,
\begin{equation}
\mcL_{\omega;A} u(\xb)= \mcDw(A(\xb)\mcGw u(\xb))= \mcL_{\widetilde\omega} u(\xb).
\end{equation}
\end{lemma}
\begin{proof}
We explicitly compute the composition of weighted divergence and gradient.
\begin{equation*}
\begin{aligned}
\mcDw(A(\xb)\mcGw u(\xb)) 
&= \mcD(\omega(\xb,\yb)A(\xb)\mcGw u(\xb)) \\   
&=\intl_\mbRn \left[\omega(\xb,\yb)A(\xb)\mcGw u(\xb)+
\omega(\yb,\xb) A(\yb) \mcGw u(\yb)\right] \MG{\cdot} \alphab(\xb,\yb)d\yb\\
&=\intl_\mbRn \MG{\Bigg[}  \omega(\xb,\yb) A(\xb) \intl_\mbRn \MG{\omega(\xb,\zb)} \mcG u(\xb,\zb) d\zb \MG{\Bigg]} \MG{\cdot \alphab(\xb,\yb) \,d\yb}\\
&\phantom{=}+\intl_\mbRn  \MG{\Bigg[} \omega(\yb,\xb) A(\yb) \intl_\mbRn \MG{\omega(\yb,\zb)} \mcG u(\yb,\zb)d\zb \MG{\Bigg]} \MG{\cdot \alphab(\xb,\yb) \,d\yb}\\
&=\MG{\intl_\mbRn \Bigg[\omega(\xb,\yb)  A^{\frac{1}{2}}(\xb) \intl_\mbRn  
\omega(\xb,\zb) A^{\frac{1}{2}}(\xb) 
\mcG u(\xb,\zb)
d\zb \Bigg] \cdot \alphab(\xb,\yb) \,d\yb}\\
&\phantom{=}\MG{+\intl_\mbRn  \Bigg[ \omega(\yb,\xb) A^{\frac{1}{2}}(\yb) \intl_\mbRn 
\omega(\yb,\zb) A^{\frac{1}{2}}(\yb)
\mcG u(\yb,\zb)
d\zb \Bigg] \cdot \alphab(\xb,\yb)\,d\yb} \\
&=\MG{\intl_\mbRn \Bigg[\widetilde\omega(\xb,\yb)   \intl_\mbRn  
\widetilde\omega(\xb,\zb) 
\mcG u(\xb,\zb)
d\zb \Bigg] \cdot \alphab(\xb,\yb) \,d\yb}\\
&\phantom{=}\MG{+\intl_\mbRn  \Bigg[ \widetilde\omega(\yb,\xb) \intl_\mbRn  
\widetilde\omega(\yb,\zb)
\mcG u(\yb,\zb)
d\zb \Bigg] \cdot \alphab(\xb,\yb)\,d\yb} \\
&=
\MG{\mcD_{\widetilde\omega}\mcG_{\widetilde\omega}u(\xb),
\quad \text{ by \eqref{eq:w-lapl}}.} 
\end{aligned}
\end{equation*}
.
\hspace{14cm}$\square$
\end{proof}

Having established that the addition of a space-dependent diffusion tensor corresponds to having a nonsymmetric weight function in the nonlocal Laplacian operator \eqref{eq:w-lapl}, we show that the corresponding weighted Laplacian still admits a symmetric equivalence kernel. Note that the arguments below hold also when $\omega$ is a tensor. 
\begin{lemma}\label{anisotropic-equivalence}
Let the weight function $\omega$ be two-point function, not necessarily symmetric, i.e. $\omega(\xb,\yb)\neq\omega(\yb,\xb)$. Then, there exists a symmetric equivalence kernel $\gameq$ such that
\begin{equation}\label{eq:A-eq-kernel}
\mcDw(\mcGw u(\xb)) = 
2\int_{\mathbb R^n} (u(\yb)-u(\xb))\gameq(\xb,\yb;\omega)\,d\xb,
\end{equation}
where $\gameq(\xb,\yb;\omega,\alphab)$ is a symmetric function of $\xb$ and $\yb$.
\end{lemma}
\begin{proof}
First, we derive the equivalence kernel corresponding to a nonsymmetric weight.
\begin{align}
   \mcDw\mcGw u(\xb) 
&= \intl_\mbRn (\omega(\xb,\yb)\mcGw u(\xb) + \omega(\yb,\xb)\mcGw u(\yb))\cdot
   \alphab(\xb,\yb)d\yb \notag\\
&= \intl_\mbRn \Bigg[\omega(\xb,\yb) \intl_\mbRn \MG{\omega(\xb,\zb)} (u(\zb)-u(\xb))
   \alphab(\xb,\zb)d\zb\notag\\ 
&\qquad\qquad + \omega(\yb,\xb)\intl_\mbRn \MG{\omega(\yb,\zb)} (u(\zb)-u(\yb))
   \alphab(\yb,\zb)d\zb \Bigg] \cdot 
   \alphab(\xb,\yb)d\yb \notag\\
&= \intl_\mbRn \intl_\mbRn (u(\zb)-u(\xb))
   \MG{\omega(\xb,\yb) \omega(\xb,\zb)}
   \alphab(\xb,\zb)\cdot 
   \alphab(\xb,\yb)d\yb d\zb  \label{eq:eqnA}\\ 
&\quad+ \intl_\mbRn \intl_\mbRn (u(\zb)-u(\yb))
\MG{\omega(\yb,\xb) \omega(\yb,\zb)}
   \alphab(\yb,\zb)\cdot
   \alphab(\xb,\yb)d\yb d\zb. \label{eq:eqnB}
\end{align}
Let the integral in \eqref{eq:eqnA} be $I$ and the one in \eqref{eq:eqnB} be $II$. We have\MG{, by self-adjointness of $\omega(\xb,\yb)$,}
\begin{align*}
I 
&= \intl_\mbRn \intl_\mbRn (u(\zb)-u(\xb))
   \MG{\omega(\xb,\zb)}\alphab(\xb,\zb)\cdot 
   \MG{\omega(\xb,\yb)}\alphab(\xb,\yb)d\yb d\zb \\
&= \intl_\mbRn (u(\zb)-u(\xb))\MG{\omega(\xb,\zb)}\alphab(\xb,\zb) \cdot
   \intl_\mbRn \MG{\omega(\xb,\yb)}\alphab(\xb,\yb)d\yb d\zb\\
&= \intl_\mbRn (u(\zb)-u(\xb)) \gamma_I(\xb,\yb)d\zb,
\end{align*}
where we have defined $\gamma_I(\xb,\zb) = \MG{\omega(\xb,\zb)} \alphab(\xb,\zb) \cdot \intl_\mbRn \MG{\omega(\xb,\yb)} \alphab(\xb,\yb)d\yb$.
Next, \MG{by self-adjointness of $\omega(y,x)$,}
\begin{align*}
II &= \intl_\mbRn \intl_\mbRn (u(\zb)-u(\yb))
     \MG{\omega(\yb,\zb)} \alphab(\yb,\zb) \cdot
     \MG{\omega(\yb,\xb)} \alphab(\xb,\yb)d\yb d\zb \\
  &= \intl_\mbRn \intl_\mbRn (u(\zb)-u(\xb))
     \MG{\omega(\yb,\zb)}\alphab(\yb,\zb)\cdot
     \MG{\omega(\yb,\xb)}\alphab(\xb,\yb)d\yb d\zb \\ 
  &\phantom{=} + \intl_\mbRn \intl_\mbRn (u(\xb)-u(\yb))
     \MG{\omega(\yb,\zb)} \alphab(\yb,\zb)\cdot
     \MG{\omega(\yb,\xb)} \alphab(\xb,\yb)d\yb d\zb. 
\end{align*}
Switching $\yb$ and $\zb$ in the first integral, and employing the anti-symmetry of $\alphab$, we find
\begin{align*}
II &= \intl_\mbRn \intl_\mbRn (u(\zb)-u(\xb))
     \MG{\omega(\yb,\zb)} \alphab(\yb,\zb)\cdot
     \MG{\omega(\yb,\xb)}\alphab(\xb,\yb)d\yb d\zb \\ 
  &\phantom{=} + \intl_\mbRn \intl_\mbRn (u(\xb)-u(\zb))
     \MG{\omega(\zb,\yb)} \alphab(\zb,\yb)\cdot
     \MG{\omega(\zb,\xb)} \alphab(\xb,\zb)d\xb d\zb\\
  &= \intl_\mbRn \intl_\mbRn (u(\zb)-u(\xb))
     \MG{\omega(\yb,\zb)} \alphab(\yb,\zb)\cdot
     \MG{\omega(\yb,\xb)} \alphab(\xb,\yb)d\yb d\zb \\ 
  &\phantom{=} + \intl_\mbRn \intl_\mbRn (u(\zb)-u(\xb))
     \MG{\omega(\zb,\yb)} \alphab(\yb,\zb)\cdot
     \MG{\omega(\zb,\xb)} \alphab(\xb,\zb)d\zb d\yb\\
  &= \intl_\mbRn (u(\zb)-u(\xb)) \intl_\mbRn
     [\MG{\omega(\yb,\zb)} \alphab(\yb,\zb) \cdot \MG{\omega(\yb,\xb)} \alphab(\xb,\yb)
     +\MG{\omega(\zb,\yb)} \alphab(\yb,\zb) \cdot \MG{\omega(\zb,\xb)} \alphab(\xb,\zb)]d\yb d\zb \\
  &= \intl_\mbRn (u(\zb)-u(\xb)) 
  \MG{\gamma_{II}(\xb,\zb) d\zb}.
\end{align*}
\MG{Above, we have put}
$$
\begin{aligned}
\gamma_{II}(\xb,\zb)
&=\intl_\mbRn
     [\MG{\omega(\yb,\zb)}\alphab(\yb,\zb) \cdot \MG{\omega(\yb,\xb)} \alphab(\xb,\yb)
     +\MG{\omega(\zb,\yb)} \alphab(\yb,\zb) \cdot \MG{\omega(\zb,\xb)} \alphab(\xb,\zb)]d\yb \\
&=\MG{\intl_\mbRn
     \MG{\omega(\yb,\zb)}\alphab(\yb,\zb) \cdot \MG{\omega(\yb,\xb)} \alphab(\xb,\yb) d\yb
     +\MG{\omega(\zb,\xb)} \alphab(\xb,\zb) \cdot \intl_\mbRn
     \MG{\omega(\zb,\yb)} \alphab(\yb,\zb) d\yb}\\
     &= \gamma_{II_A}(\xb,\zb)+\gamma_{II_B}(\xb,\zb).\\
\end{aligned}
$$
We next show that $\gamma^*(\xb,\zb)=\gamma_I(\xb,\zb)+\gamma_{II}(\xb,\zb)$ is symmetric. \MG{Using antisymmetry of $\alpha$,}
\begin{align*}
\gamma^*(\zb,\xb) 
&= 
\begin{multlined}[t]\MG{\omega(\zb,\xb)} \alphab(\zb,\xb) \MG{\cdot} \intl_\mbRn
   \MG{\omega(\zb,\yb)} \alphab(\zb,\yb)d\yb 
   + 
   \intl_\mbRn \MG{\omega(\yb,\xb)} \alphab(\yb,\xb) \MG{\cdot}
   \MG{\omega(\yb,\zb)}\alphab(\zb,\yb) \MG{d\yb} 
   \\
+ \MG{\omega(\xb,\zb)  }\alphab(\zb,\xb) \MG{\cdot}
   \intl_\mbRn\MG{\omega(\xb,\yb)} \alphab(\yb,\xb)d\yb
   \end{multlined}
   \\
   &= \begin{multlined}[t]
   \MG{\MG{\omega(\zb,\xb)} \alphab(\xb,\zb) \MG{\cdot} \intl_\mbRn
   \MG{\omega(\zb,\yb)} \alphab(\yb,\zb)d\yb 
   + 
      \intl_\mbRn \MG{\omega(\yb,\xb)} \alphab(\xb,\yb) \MG{\cdot}
   \MG{\omega(\yb,\zb)}\alphab(\yb,\zb) \MG{d\yb} }
   \\
   \MG{
+ \MG{\omega(\xb,\zb)  }\alphab(\xb,\zb) \MG{\cdot}
   \intl_\mbRn\MG{\omega(\xb,\yb)} \alphab(\xb,\yb)d\yb}
   \end{multlined}
   \\
   &= \begin{multlined}[t]
   \MG{\MG{\omega(\zb,\xb)} \alphab(\xb,\zb) \MG{\cdot} \intl_\mbRn
   \MG{\omega(\zb,\yb)} \alphab(\yb,\zb)d\yb 
   + 
    \intl_\mbRn 
   \MG{\omega(\yb,\zb)}\alphab(\yb,\zb) 
   \MG{\cdot}
   \MG{\omega(\yb,\xb)} \alphab(\xb,\yb) \MG{d\yb}} 
   \\
   \MG{
+ \MG{\omega(\xb,\zb)  }\alphab(\xb,\zb) \MG{\cdot}
   \intl_\mbRn\MG{\omega(\xb,\yb)} \alphab(\xb,\yb)d\yb}
   \end{multlined}
   \\
&= \MG{\gamma_{II_B}(\xb,\zb) + \gamma_{II_A}(\xb,\zb) +  \gamma_I(\xb,\zb)}\\
& = \gamma^*(\xb,\zb).
\end{align*}
Then, \eqref{eq:A-eq-kernel} follows by setting $\gameq=2\gamma^*$. \hspace{8cm}$\square$
\end{proof}

We now introduce the anisotropic nonlocal Poisson equation. For $f:\Omega\to\mbR$, we seek $u:\mbRn\to\mbR$ such that
\begin{equation}\label{eq:elliptic-anisotropic}
\left\{\begin{aligned}
-\mcL_{\omega;A} u(\xb,t) &= f(\xb,t), & \quad \xb\in\Omega\\
u(\xb)& = 0,        & \quad \xb\in\mbRn\setminus\Omega.
\end{aligned}\right.
\end{equation}
As usual, a form of Green's first identity is required to introduce a weak form for the equation above. The next theorem extends the weighted nonlocal Green's identity to the anisotropic case. 
\begin{theorem}\label{anisotropicGreen}
Let $\mcL_{\omega;A}$ be defined as in \eqref{eq:ALaplacian}. Then,
\begin{equation}\label{eq:anisotropicGreen}
-\intl_\Omega \mcL_{\omega;A} u(\xb) v(\xb)\ d\xb = 
\intl_\mbRn \MG{\mcGw v(\xb)\cdot} A(\xb)\mcGw u(\xb)  d\xb
+ \intl_{\mbRn\setminus\Omega}\mcDw(A(\xb)\mcGw  u(\xb)) v(\xb) \ d\xb
\end{equation}
\end{theorem}
\begin{proof}
We first note that by combining the left-hand side of \eqref{eq:anisotropicGreen} and the second term on the right-hand side of the same equation, we have:
\begin{equation*}
\intl_\Omega \mcL_{\omega;A} u(\xb) v(\xb)\ d\xb +
\intl_{\mbRn\setminus\Omega}\mcDw(A(\xb)\mcGw  u(\xb)) v(\xb) \ d\xb  = 
\intl_\mbRn\mcDw(A(\xb)\mcGw  u(\xb)) v(\xb) \ d\xb. 
\end{equation*}
We compute explicitly the right-hand side of the equation above, using the definition of weighted operators,
\begin{equation*}
\begin{aligned}
\intl_\mbRn\mcDw(A(\xb)\mcGw  u(\xb))& v(\xb)\ d\xb
=
\intl_\mbRn\mathcal{D}\left(\omega(\xb,\yb)A(\xb)\mathcal{G}_\omega u(\xb)\right)v(\xb) \ d\xb\\
&=
\intl_\mbRn\intl_\mbRn\left[\omega(\xb,\yb)A(\xb) \mcGw u(\xb) +\omega(\yb,\xb) A(\yb)\mcGw u(\yb) \right] \cdot
\alphab(\xb,\yb) v(\xb) \ d\yb d\xb\\
&=
\intl_\mbRn\intl_\mbRn\omega(\xb,\yb) 
\left[\;\intl_\mbRn A(\xb) \MG{\omega(\xb,\zb)} \mcG u(\xb,\zb) d\zb+
\intl_\mbRn A(\yb) \MG{\omega(\yb,\zb)} \mcG u(\yb,\zb) d\zb\right]\\
&\hspace{8cm}\cdot\alphab(\xb,\yb)v(\xb) \ d\yb d\xb\\[2mm]
&=
\intl_\mbRn\intl_\mbRn\intl_\mbRn
\omega(\xb,\yb)  A(\xb) \MG{\omega(\xb,\zb)} \mcG u(\xb,\zb)
\cdot\alphab(\xb,\yb)v(\xb)  d\zb d\yb d\xb\\
&\phantom{=} +\intl_\mbRn\intl_\mbRn\intl_\mbRn
\omega(\xb,\yb) A(\yb) \MG{\omega(\yb,\zb)} \mcG u(\yb,\zb)  
\cdot \alphab(\xb,\yb) v(\xb)  d\zb d\yb d\xb.
\end{aligned}
\end{equation*}
Applying the change of variables $x \mapsto y \mapsto z \mapsto x$, 
\begin{multline}
\intl_\mbRn\mcDw(A(\xb)\mcGw  u(\xb)) v(\xb)\ d\xb
=
\intl_\mbRn\intl_\mbRn\intl_\mbRn
\omega(\xb,\yb)  A(\xb) \MG{\omega(\xb,\zb)} 
\mcG u(\xb,\zb) 
\cdot\alphab(\xb,\yb)v(\xb)  d\zb d\yb d\xb\\
+
\intl_\mbRn\intl_\mbRn\intl_\mbRn
\omega(\yb,\zb) A(\zb) \MG{ \omega(\zb,\xb)} \mcG u(\zb,\xb) 
\cdot \alphab(\yb,\zb) v(\yb)  d\xb d\zb d\yb.
\end{multline}
By using the definition of the weighted gradient \MG{and self-adjointness of $\omega(x,y)$}, we have
\begin{equation*}
\begin{aligned}
\intl_\mbRn\mcDw(A(\xb)\mcGw  u(\xb)) v(\xb) \ d\xb
&=\intl_\mbRn\intl_\mbRn\omega(\xb,\yb)A(\xb) \mcGw u(\xb)\MG{\cdot}\alphab(\xb,\yb)v(\xb)d\yb d\xb\\
&\phantom{=} + \intl_\mbRn\intl_\mbRn \omega(\yb,\zb) A(\zb) \mcGw u(\zb) \MG{\cdot} \alphab(\yb,\zb) v(\yb)d\zb d\yb\\
& = \intl_\mbRn\intl_\mbRn\omega(\xb,\yb)A(\xb) \mcGw u(\xb)\MG{\cdot}\alphab(\xb,\yb)v(\xb)d\yb d\xb\\
&\phantom{=} +\intl_\mbRn\intl_\mbRn \omega(\yb,\xb) A(\xb) \mcGw u(\xb) \MG{\cdot}\alphab(\yb,\xb) v(\yb)d\xb d\yb\\
&=\intl_\mbRn\intl_\mbRn\omega(\xb,\yb)A(\xb) \mcGw u(\xb) \MG{\cdot}\alphab(\xb,\yb) [v(\xb)-v(\yb)]d\xb\,d\yb\\
&= -\intl_\mbRn \MG{[}A(\xb) \mcGw u(\xb) \MG{]} \MG{\cdot}
\intl_\mbRn \omega(\xb,\yb)\alphab(\xb,\yb) [v(\yb)-v(\xb)]d\xb\,d\yb\\
&=- \intl_\mbRn \MG{\mcGw v(\xb) \cdot } A(\xb)\mcGw u(\xb)  d\xb.
\end{aligned}
\end{equation*}
\hspace{14cm}$\square$
\end{proof}

\subsection{Weak form of anisotropic Poisson problem}

Utilizing the results of the previous subsection, we can formulate the weak form of equation \eqref{eq:elliptic-anisotropic} and show that the corresponding energy is equivalent to an unweighted nonlocal energy. 
We multiply \eqref{eq:elliptic-anisotropic} by a test function $v=0$ in $\mbRn\setminus\Omega$ and integrate over the domain $\Omega$; we have
\begin{equation}\label{eq:weight-truncated-weak-A}
\int_\Omega (-\mcL_{\omega;A} u\MG{(\xb)} -f\MG{(\xb)})\, v\MG{(\xb)} \,d\xb = 0.
\end{equation}
The anisotropic weighted nonlocal Green's first identity \eqref{eq:anisotropicGreen} then implies
\begin{displaymath}
\int_\mbRn \!   \MG{\mcGw u(\xb) \cdot} A(\xb) \mcGw u(\xb) d\xb 
+ \int_{\mbRn \setminus \Omega} \!\mcDw(A(\xb)\mcGw u(\xb)) v(\xb) \ d\xb
- \int_\Omega f\MG{(\xb)}\, v\MG{(\xb)} \,d\xb=0.
\end{displaymath}
Thus, the weak form of the nonlocal Poisson problem reads as follows. For $f\in V_A'(\mbRn)$, find $u\in V^A_\Omega(\mbRn)$ such that
\begin{equation}\label{eq:weighted-weak-A}
\mcB_{\omega;A}(u,v) = \mcF(v), \;\;\forall\, v\in V_A(\omgomgi),
\end{equation}
where 
\begin{equation}\label{eq:A-F-tensor}
\begin{aligned}
\mcB_{\omega;A}(u,v) & = \int_\mbRn \MG{\mcGw u(\xb) \, \cdot \,} A(\xb) \mcGw u(\xb)   d\xb,\\[2mm]
V^A_\Omega(\mbRn) & =\{v\in L^2(\mbRn): \vertiii{v}_A<\infty\;
                {\rm and} \; v|_{\mbRn\setminus\Omega} = 0\},
\end{aligned}
\end{equation}
and where the {\it anisotropic energy} is defined as
\begin{equation}\label{eq:weighted-energy-A}
\vertiii{v}_A^2 = \int_\mbRn \mcGw v\MG{(\xb)} \cdot A(\xb)\mcGw v\MG{(\xb)}\,d\xb.
\end{equation}
The existence of the equivalence kernel guaranteed by Lemma \ref{anisotropic-equivalence}, allows us to establish an equivalence relationship between the anisotropic weighted bilinear form $\mcB_{\omega;A}$ defined above and the unweighted bilinear form $\mcB$ given in \eqref{eq:bound-truncated-weak-forms}, where the latter is associated to the equivalence kernel $\gameq(\xb,\yb;A^\frac12\omega,\alphab)$, as shown in the following lemma. 
\begin{lemma}
Let $A$ be a bounded, measurable and elliptic tensor, $\omega$ be a symmetric scalar function and $\alphab$ an anti-symmetric vector function. Then, the following identity holds:
\begin{equation}
\mcB_{\omega;A}(u,v)
=\int_\mbRn \MG{ \mcGw v\MG{(\xb)}\cdot} A(\xb)\mcGw u\MG{(\xb)}  \,d\xb=\mcB(u,v), \quad \forall\, u,v\in V^A_\Omega(\mbRn),
\end{equation}
where $\mcB(\cdot,\cdot)$ is the unweighted bilinear form defined in \eqref{eq:A-F-V} associated to the symmetric equivalence kernel $\gameq(\xb,\yb;A^\frac12\omega,\alphab)$.
\end{lemma}
\begin{proof}
The proof follows from Lemmas \ref{anisotropic-equivalence} and \ref{anisotropicGreen}. We have
\begin{displaymath}
\begin{aligned}
\int_\mbRn \MG{ \mcGw v(\xb) \cdot} A(\xb)  \mcGw u\MG{(\xb)}  \,d\xb
& = - \int_\omg \mcDw(A(\xb)\mcGw u)\MG{(\xb)} v\MG{(\xb)} \,d\xb 
&     \quad\hbox{(weighted Green's identity \eqref{anisotropicGreen})}\\[2mm]
& = - \int_\omg \mcD\mcG u\MG{(\xb)} v\MG{(\xb)} \,d\xb
&     \quad\hbox{(Equivalence kernel; Theorem \ref{anisotropic-equivalence})} \\[2mm]
& =   \MG{\int_{\mbRn}\int_{\mbRn}}
      \mcG u\MG{(\xb,\yb)}\, \MG{\cdot} \mcG v\MG{(\xb,\yb)} \,d\yb\,d\xb . 
&     \quad\hbox{(unweighted Green's identity \eqref{eq:unweighted-Green})} \\[2mm]
\end{aligned}
\end{displaymath} 
\hspace{14cm}$\square$
\end{proof}
The theorem above is not enough to guarantee the well-posedness of problem \eqref{eq:weighted-weak-A}. One approach to obtaining the existence and uniqueness of solutions involves establishing certain properties of $\gameq$, as highlighted in \cite[Section~5]{DElia2020Unified}. However, thanks to the ellipticity property of $A$, the well-posedness of the anisotropic problem follows from the well-posedness of the weighted problem associated with the corresponding isotropic weighted bilinear form $\mcB_\omega$. In fact, $\mcB_{\omega;A}$ is coercive and continuous with respect to the energy induced by of $\mcB_\omega$, as we show in the following lemma.
\begin{lemma}\label{eq:lem-cont-coer-A}
The bilinear form $\mcB_{\omega;A}(u,v)$ defined as in \eqref{eq:A-F-tensor} is continuous and coercive in $V_\Omega^\omega(\mbRn)$, i.e. 
\begin{equation}\label{eq:A-continuity-coercivity}
\begin{aligned}
|\mcB_{\omega;A}(u,v)|& \leq \lambda_{\rm max}\, \vertiii{u}_\omega \vertiii{v}_\omega\\
\mcB_{\omega;A}(u,u)&\geq \lambda_{\rm min}\vertiii{u}^2_\omega,
\end{aligned}
\end{equation}
where $\lambda_{\rm min}$ and $\lambda_{\rm max}$ are the smallest and largest eigenvalues of $A$ over $\mbRn$, respectively.
\end{lemma}
\begin{proof}
The ellipticity of $A$ allows us to write
$$
\begin{aligned}
|\mcB_{\omega;A}(u,v)| 
& = \int_\mbRn (A^\frac12\mcGw u \MG{(\xb)}) \MG{\cdot} 
(A^\frac12\mcGw v \MG{(\xb)})\,d\xb
\\
& \MG{=\int_\mbRn |\,A^\frac12\mcGw u (\xb)| \,
|\, A^\frac12\mcGw v (\xb) |\,d\xb}
\\
& \leq \left(\int_\mbRn \MG{|} \, A\MG{^\frac{1}{2}}\mcGw u\MG{(\xb)}\MG{|}^2  \; \int_\mbRn \MG{|} \, A\MG{^\frac{1}{2}}\mcGw v\MG{(\xb)}\MG{|}^2\right)^\frac12 \\
& \leq \lambda_{\rm max}\, \vertiii{u}_\omega \vertiii{v}_\omega,
\end{aligned}
$$
which implies, by definition, the continuity of $\mcB_{\omega;A}$ with respect to the norm $\vertiii{\cdot}_\omega$. The coercivity with respect to the same norm simply follows from
\begin{equation}
\mcB_{\omega;A}(u,u)=\int_\mbRn \MG{\mcGw u(\xb) \cdot} A\,\mcGw u\MG{(\xb)} \, d\xb> \lambda_{\rm min}\vertiii{u}^2_\omega.
\end{equation}
$\square$
\end{proof}
%

\subsection{Well-posedness of the anisotropic fractional Poisson equation} \label{sec:anisotropic-fractional-Poisson}
In this section, we apply the analysis of the nonlocal anisotropic problem to the case of fractional operators. That is, we consider $\omega$ and $\alphab$ defined as in \eqref{eq:equivalence-w-alpha} and show that the corresponding anisotropic problem is well-posed in the usual fractional Sobolev space. In order to do this, we only need to show that the bilinear form $\mcB_{\omega;A}$ is coercive and continuous with respect to the fractional Sobolev norm. 
In fact, Theorem \eqref{thm:fractional_special_case_nonlocal} states that the equivalence kernel associated with the weight and kernel functions in \eqref{eq:equivalence-w-alpha} is the fractional Laplacian kernel $\gamma_{F\!L}$; the variational equivalence discussed in Remark \ref{rem:fractional-forms} implies that the corresponding weighted energy space $V^\omega_\Omega$ is equivalent to $H^s_\Omega$ and that the weighted energy $\vertiii{\cdot}_\omega$ is equivalent to the $H^s$ norm. Thus, Lemma \ref{eq:lem-cont-coer-A} implies the continuity and coercivity of $\mcB_{\omega;A}$ in $H^s_\Omega$ and the well-posedness of problem \eqref{eq:weighted-weak-A} is immediate, as stated in the folling lemma.
\begin{lemma}
Let $A$ satisfy \eqref{eq:Atensor}, and let $\omega$ and $\alphab$ be defined as in \eqref{eq:equivalence-w-alpha}. Then, the corresponding bilinear form $\mcB_{\omega;A}$ defined as in \eqref{eq:A-F-tensor} is coercive and continuous in $H^s_\Omega(\mbRn)$ with coercivity and continuity constants
\begin{equation}\label{eq:coer-cont-AA}
C_{\rm coer}=\dfrac{C_{n,s}}{2}\lambda_{\rm min} 
\quad {\rm and} \quad
C_{\rm cont}=\dfrac{C_{n,s}}{2}\lambda_{\rm max},
\end{equation}
where $\lambda_{\rm min}$ and $\lambda_{\rm max}$ are the smallest and largest eigenvalues of $A$ in $\mbRn$, respectively.
Furthermore, problem \eqref{eq:weighted-weak-A} is well-posed.
\end{lemma}

Note that for the fractional case and for a class of tensors satisfying \eqref{eq:Atensor} that we specify below, we can characterize the equivalence kernel. In particular, the equivalence kernel is such that the corresponding unweighted bilinear form is a {\it Dirichlet form}, as we show in the following lemma. 
\begin{lemma}\label{lem:isotropic-characterization}
Let $I$ be the identity tensor in $\mbRn$ and let $A(\xb)=a(\xb)I$ satisfy \eqref{eq:Atensor} for $a:\mbRn\to\mbR$, i.e. there exist two positive constants such that $0<\underline{a}\leq a(\xb)\leq \overline{a}<\infty$. Then, the equivalence kernel $\gameq(\xb,\yb;a^\frac12\omega,\alphab)$ is such that
\begin{equation*}
\dfrac{\underline{a}C_{n,s}}{2}\,
|\xb-\yb|^{-n-2s}
\leq \gameq(\xb,\yb;a^\frac12\omega,\alphab)\leq 
\dfrac{\overline{a}C_{n,s}}{2}\, 
|\xb-\yb|^{-n-2s}.
\end{equation*}
\end{lemma}
Lemma \ref{lem:isotropic-characterization} implies that the equivalence kernel is positive; in addition to symmetry, this property guarantees that the corresponding unweighted bilinear form $\mcB$ is a Dirichlet form \cite{Fukushima1980}. We point out that the class of tensors in Lemma \ref{lem:isotropic-characterization} corresponds to a space dependent isotropic diffusion as the intensity of the diffusion is the same in all directions. 

\begin{remark}
It is unclear how to characterize the class of equivalence kernels for a general anisotropic fractional-order operator. In our terminology of symmetric equivalence kernels, \cite{Shieh2017} poses the problem of whether or not the equivalence kernel corresponding to anisotropic fractional-order operators with a tensor $A(\xb)$ satisfying \eqref{eq:Atensor} must satisfy the following conditions:
\begin{equation}\label{eq:gamma-conditions}
\begin{aligned}
  \lambda\leq \gameq(\xb,\yb)|\xb-\yb|^{n+2r}\leq \Lambda 
& \quad |\xb-\yb|\leq 1, \\
  \gameq(\xb,\yb)|\xb-\yb|^{n+2t}\leq M 
& \quad  |\xb-\yb|> 1,
\end{aligned}
\end{equation}
for $r\in(0,1)$, $0<\lambda\leq\Lambda<\infty$, $M<\infty$ and $t>0$. Lemma \ref{lem:isotropic-characterization} shows that the conditions in \eqref{eq:gamma-conditions} are satisfied in the isotropic case. The question remains of whether those bounds hold for any tensor satisfying \eqref{eq:Atensor}.
\end{remark}

\section{Well-posedness of a parabolic equation with anisotropic nonlocal diffusion}\label{sec:anisotropic-parabolic}
The results of Section \ref{sec:anisotropic-Poisson} allow us to analyze the anisotropic parabolic problem. In fact, the coercivity of the bilinear form $\mcB_{\omega;A}$ implies the well-posedness of the corresponding parabolic problem, for which weak coercivity would be sufficient. 
We introduce the strong form of the anisotropic parabolic equation and, by using the Green's identity \eqref{eq:anisotropicGreen}, we formulate the corresponding weak problem and state a well-posedness result. 

For $f:\Omega\to\mbR$ and $u_0:\Omega\to\mbR$, we seek $u$ such that
\begin{equation}\label{eq:parabolic-anisotropic}
\left\{\begin{aligned}
\partial_t u(\xb,t) &= \mcDw(A(\xb)\mcGw u(\xb,t)) + f(\xb,t),  
& \quad (\xb,t)\in\Omega\times(0,T]\\
u(\xb,t) &= 0,       
&\quad (\xb,t)\in\mbRn\setminus\Omega\times(0,T]\\
u(\xb,0) &= u_0(\xb),  
&\quad \xb\in\Omega
\end{aligned}\right.
\end{equation}
By multiplying \eqref{eq:parabolic-anisotropic} by a test function  $v=0$ in $\mbRn\setminus\Omega$, integrating over the domain $\Omega$, and using the anisotropic Green's identity \eqref{eq:anisotropicGreen}, we have
\begin{equation}\label{eq:anisotropic-parabolic-weak}
\begin{aligned}
0&=\int_\Omega (\partial_t u\MG{(\xb,t)}-\mcLw u\MG{(\xb,t)} -f\MG{(\xb,t)})\, v\MG{(\xb)} \,d\xb\\
&= \int_\Omega \partial_t u\MG{(\xb,t)}\,v\MG{(\xb)} \,d\xb +
\int_\mbRn \MG{\mcGw u(\xb)  \cdot} A(\xb)\mcGw u(\xb) d\xb 
- \int_\Omega f\MG{(\xb,t)}\, v\MG{(\xb)} \,d\xb.
\end{aligned}
\end{equation}
Thus, the weak form of the nonlocal diffusion problem reads as follows. For $f\in L^2(0,T;(V^A_\Omega)'(\mbRn))$, and $u_o\in V^A_\Omega(\mbRn)$, find $u\in L^2(0,T;V^A_\Omega(\mbRn)$ such that
\begin{equation}\label{eq:anisotropic-parabolic-weakA}
(\partial_t u,v)+ \mcB_{\omega;A}(u,v) = \mcF(v), \;\;\forall\, v\in V_\Omega(\mbRn).
\end{equation}
When the equivalence kernel $\gameq(\xb,\yb;A^\frac12\omega)$ associated with $A$ is such that the unweighted bilinear form $\mcB$ is coercive, problem \eqref{eq:anisotropic-parabolic-weakA} is well-posed, as we state in the following theorem.
\begin{theorem}\label{thm:aniso-parabolic-wellp}
For $f\in  L^2(0,T;V'_\Omega(\mbRn))$, $u_0\in V^A_\Omega$ and $\mcB_{\omega;A}(\cdot,\cdot)$ such that the corresponding $\gameq(\xb,\yb;A^\frac12\omega,\alphab)$ induces a weakly coercive and continuous unweighted for $\mcB(\cdot,\cdot)$, the problem \eqref{eq:anisotropic-parabolic-weakA} has a unique solution $u^*\in  L^2(0,T;V_\Omega(\mbRn))$, where $V_\Omega(\mbRn)$ is the energy space associated with the bilinear form $\mcB(\cdot,\cdot)$. 

Furthermore, if $\mcB(\cdot,\cdot)$ is coercive and the associated energy norms satisfies a Poincar\'e inequality with constant $C_p$, that solution satisfies the a priori estimate
\begin{equation}\label{eq:a-priori}
\|u^*(\cdot,t)\|^2_{L^2(\Omega)}+ C_{\rm coer}\int_0^t \vertiii{u^*(\cdot,s)}^2\,ds  \leq \|u_0\|^2_{L^2(\Omega)} +\dfrac{C_p^2}{2C_{\rm coer}}\int_0^t \|f(\cdot,s)\|^2_{V'_\Omega}\,ds  \quad  \forall \; t>0,
\end{equation}
where $\|\cdot\|_{V'_\Omega}$ indicates the standard operator norm in the dual space of $V_\Omega(\mbRn)$ and $C_{\rm coer}$ is the coercivity constant of the bilinear form $\mcB(\cdot,\cdot)$.
\end{theorem}

\begin{proof}
The conditions on $f$ guarantee the continuity of the functional $\mcF(\cdot)$. The weak coercivity and the continuity of $\mcB(\cdot,\cdot)$ and the continuity of $\mcF(\cdot)$ imply the existence and uniqueness of a solution $u^*\in L^2(0,T;V_\Omega(\mbRn))$ \cite{Evans1998}. Then, \eqref{eq:a-priori} follows from arguments entirely similar to those used in the classical theory of partial differential equations \cite{Salsa2008}. $\square$
\end{proof}
The results of Section \ref{sec:anisotropic-fractional-Poisson} show that for a tensor $A$ satisfying \eqref{eq:Atensor}, and for $\omega$ and $\alphab$ as in \eqref{eq:equivalence-w-alpha}, the equivalence kernel associated with $\mcDw(A(\xb)\mcGw)$ induces an unweighted bilinear form $\mcB(\cdot,\cdot)$, whose energy norm is equivalent to the $H^s$-norm. This implies that $\mcB(\cdot,\cdot)$ is coercive and continuous on $H^s_\Omega(\mbRn)$.
Thus, Theorem \ref{thm:aniso-parabolic-wellp} can be immediately applied to the special case of fractional operators, as we show in the following corollary. Note that, in this case, the unweighted energy norm corresponds to the $H^s$ norm for which the Poincar\'e inequality is satisfied for all $u\in H^s_\Omega(\mbRn)$ \cite{Burkovska2020affine}.
\begin{corollary}\label{coroll:wellposedness}
Let $A$ be a tensor satisfying \eqref{eq:Atensor}, and let $\omega$ and $\alphab$ be defined as in \eqref{eq:equivalence-w-alpha}. For $f\in  L^2(0,T;(H^s_\Omega(\mbRn))')$ and $u_0\in H^s_\Omega(\mbRn)$, the problem
\begin{equation}\label{eq:anisotropic-parabolic-weak-frac}
(\partial_t u,v)+ \mcB_{\omega;A}(u,v) = \mcF(v), \;\;\forall\, v\in H^s_\Omega(\mbRn),
\end{equation}
has a unique solution $u^*\in  L^2(0,T;H^s_\Omega(\mbRn))$ that satisfies the estimate \eqref{eq:a-priori} for $\vertiii{\cdot}=\|\cdot\|_{H^s_\Omega}$ and $C_{\rm coer}$ as in \eqref{eq:coer-cont-AA}.
\end{corollary}

\section{Anisotropic advection-diffusion for solute transport}\label{sec:solute}
In this section, we introduce a mathematical model that is suitable for solute transport applications and we prove its well-posedness. 
The anisotropic, fractional-order model introduced in the previous section can be used to describe surface or subsurface anisotropic anomalous diffusion of solutes (e.g. pollutants). As already mentioned in the introduction, several fractional models have been introduced in the literature with this purpose, see, e.g. \cite{Deng2004} for surface transport and \cite{Benson2001} for subsurface transport. The novelty of model proposed in this section is the characterization of the anisotropic behavior via a diffusion tensor, similarly to what is commonly done in the PDE setting. 

We point out that a diffusion tensor $A(\xb)$ can be easily added to an unweighted nonlocal Laplacian, just by defining the kernel as $\gamma(\xb,\yb)=\alphab(\xb,\yb)(A(\xb)\alphab(\xb,\yb))$; however, this definition compromises the symmetry of the kernel and, hence, the well-posedness of the associated diffusion equation. The formulation below instead guarantees that the operator $\mcL_{\omega;A}$, defined as in \eqref{eq:ALaplacian} and with $\omega$ and $\alphab$ defined as in \eqref{eq:equivalence-w-alpha}, is equivalent to an unweighted operator with a symmetric kernel, whose associated energy space is equivalent to $H^s_\Omega(\mbRn)$. Hence, the operator induces a well-posed diffusion problem. It should also be noted that, since $A(\cdot)$ is a one-point function, only pointwise information is required to define its value, as opposed to other unweighted diffusion models where the diffusion tensor is a two-point function \cite{Du2012}. In addition to being computationally complex and expensive to identify (compared to one-point tensors), a two-point diffusivity tensor may be difficult to interpret physically.

\subsection{The anisotropic anomalous transport equation and its well-posedness}
We extend the anisotropic fractional diffusion model introduced in Section \ref{sec:anisotropic-parabolic} to an advection-diffusion model that takes into account the presence of drift. We assume the advection field to be a given solenoidal field $\vb$; in general, such a field is the solution of Darcy's equation. 

Let $A$ be a bounded, measurable and elliptic tensor and $\vb$ be a bounded, solenoidal vector, i.e. $\|\vb\|_{L^\infty(\Omega)}\leq C_v<\infty$ and $\nabla\cdot\vb=0$. For $\omega$ and $\alphab$ defined as in \eqref{eq:equivalence-w-alpha}, $f:\Omega\to\mbR$, $g:\mbRn\setminus\Omega\to\mbR$ and $u_0:\Omega\to\mbRn$, the strong form of the anomalous transport problem is defined as follows
\begin{equation}\label{eq:anisotropic-transport}
\left\{\begin{aligned}
\partial_t u(\xb,t) &= \mcDw(A(\xb)\mcGw u(\xb,t)) - \vb(\xb)\!\cdot\!\nabla u(\xb,t) + f(\xb,t),  
& (\xb,t)\in\Omega\times(0,T]\\
u(\xb,t) &= g,       
&(\xb,t)\in\mbRn\setminus\Omega\times(0,T]\\
u(\xb,0) &= u_0(\xb),  
& \xb\in\Omega
\end{aligned}\right.
\end{equation}
The anisotropic Green's first identity allows us to write the weak formulation of \eqref{eq:anisotropic-transport}. For the sake of simplicity, we analyze the weak form for homogeneous volume-constraints, i.e., $g \equiv 0$. Due to the presence of the advection term and in accordance with the theory presented in \cite{Bonito2020}, we restrict the fractional order to $s\in[0.5,1)$, in order to guarantee the coercivity of the problem in presence of advection. 
For $s\in[0.5,1)$, $f\in L^2(0,T;H^{-s}(\mbRn))$, and $u_0\in H^s_\Omega(\mbRn)$, we seek $u\in L^2(0,T;H^s(\mbRn))$ such that
\begin{equation}\label{eq:anisotropic-transport-weakA}
(\partial_t u,v)+ \mcB_{\omega;A}(u,v) + (\vb\cdot\nabla u,v)= \mcF(v), \;\;\forall\, v\in H^s_\Omega(\mbRn),
\end{equation}
where $\mcB_{\omega;A}$ is the bilinear form defined in \eqref{eq:A-F-tensor}.

The following lemma shows that the bilinear form $\mcB_{\omega;A}(u,v)+(\vb\cdot\nabla u,v)$ is coercive. Its proof is a combination of equation \eqref{eq:A-continuity-coercivity} and \cite[Proposition~3]{Bonito2020}.
\begin{lemma}\label{lem:coercive-transport}
Let the fractional order $s\in[0.5,1)$ and $\lambda_{\rm min}$ be the smallest eigenvalue of the tensor $A$ over $\mbRn$. If the advection field $\vb$ is bounded and solenoidal, then the bilinear form $\mcB'(u,v)=\mcB_{\omega;A}(u,v)+(\vb\cdot\nabla u,v)$ is coercive. In particular, 
$$
\mcB'(u,u) = \mcB_{\omega;A}(u,u)\geq \dfrac{C_{n,s}}{2}\lambda_{\rm min} \|u\|^2_{H^s_\Omega}.
$$
\end{lemma}
Arguments similar to Corollary \ref{coroll:wellposedness} imply the well-posedness of problem \eqref{eq:anisotropic-transport-weakA}.

\section{Conclusion}\label{sec:conclusion}
We proposed and analyzed an anisotropic nonlocal equation generalizing several results of the unified nonlocal calculus introduced in \cite{DElia2020Unified}. In particular, we showed that, in presence of an anisotropic diffusion tensor, the weighted nonlocal Laplacian operator is equivalent to an unweighted Laplacian operator whose corresponding kernel is symmetric. For the same operator we also proved an anisotropic Green's first identity and showed that the corresponding bilinear form induces an anisotropic energy norm that is equivalent to its weighted and  unweighted counterparts. This result allowed us to prove the well-posedness of the associated elliptic and parabolic problems. Furthermore, thanks to the equivalence results presented in \cite{DElia2020Unified} for fractional operators, we showed that our theory holds in the special, important, case of fractional operators. 

The theory developed in the first part of this work allowed us to prove the well-posedness of an anisotropic nonlocal advection-diffusion problem. In the special case of fractional operators, and for solenoidal advection fields, such a model is suitable for the description of anisotropic, anomalous transport of solutes in heterogeneous media. Our model, as opposed to other models proposed in the literature, allows one to include a diffusion tensor in the same way as for PDEs, without compromising the symmetry or well-posedness of the variational form of the problem.
The existence of a symmetric equivalence kernel for anisotropic weighted nonlocal diffusion operators also implies that nonlocal diffusion operators are a sufficiently rich class of models capable of describing such behavior in applications.

\section*{Acknowledgments}
The authors would like to thank Prof. Abner J. Salgado for his key suggestions regarding the treatment of the anomalous transport problem, and Hayley Olson for providing valuable feedback on the equivalence results.

MD and MG are supported by the Sandia National Laboratories (SNL) Laboratory-directed Research and Development program and by the U.S. Department of Energy, Office of Advanced Scientific Computing Research under the Collaboratory on Mathematics and Physics-Informed Learning Machines for Multiscale and Multiphysics Problems (PhILMs) project. SNL is a multimission laboratory managed and operated by National Technology and Engineering Solutions of Sandia, LLC., a wholly owned subsidiary of Honeywell International, Inc., for the U.S. Department of Energy's National Nuclear Security Administration under contract {DE-NA0003525}. This paper, SAND2021-0098, describes objective technical results and analysis. Any subjective views or opinions that might be expressed in this paper do not necessarily represent the views of the U.S. Department of Energy or the United States Government. 

\bibliographystyle{plain}
\bibliography{references.bib}

\end{document}